\documentclass{ws-ijm}
\usepackage[bookmarksnumbered, plainpages]{hyperref}
\usepackage{cite}

\begin{document}

\markboth{Hemanth Saratchandran, Jiaogen Zhang, Pan Zhang}{A new higher order Yang--Mills--Higgs flow in Riemannian $4$-manifold}

%%%%%%%%%%%%%%%%%%%%% Publisher's Area please ignore %%%%%%%%%%%%%%
\catchline{}{}{}{}{}
%%%%%%%%%%%%%%%%%%%%%%%%%%%%%%%%%%%%%%%%%%%%%%%%%%%%%%%%%%%%%%%%%%%

\title{A new higher order Yang--Mills--Higgs flow on Riemannian $4$-manifolds}

\author{Hemanth Saratchandran$^1$, Jiaogen Zhang$^2$, Pan Zhang$^3$}

\address{1.School of Mathematical Sciences, University of Adelaide, Adelaide 5005, Australia; hemanth.saratchandran@adelaide.edu.au}

\address{2. School of Mathematical Sciences, University of Science and Technology of China, Hefei 230026, P.R. China; zjgmath@ustc.edu.cn}

\address{3. School of Mathematical Sciences, Anhui University, Hefei 230601, P.R. China; panzhang20100@ahu.edu.cn}

\maketitle

\begin{abstract}
{\it Abstract:}
Let $(M,g)$ be a closed Riemannian $4$-manifold and let $E$ be a vector bundle over $M$ with structure group $G$, where $G$ is a compact Lie group. In this paper, we consider a new higher order Yang--Mills--Higgs functional, in which the Higgs field is a section of $\Omega^0(\textmd{ad}E)$. We show that, under suitable conditions, solutions to the gradient flow do not hit any finite time singularities. In the case that $E$ is a line bundle, we are able to use a different blow up procedure and obtain an improvement of the long time result in \cite{Z1}. The proof is rather relevant to the properties of the Green function, which is very different from the previous techniques in \cite{Ke,Sa,Z1}.
\end{abstract}


\ccode{2020 Mathematics Subject Classification. 58C99; 58E15; 81T13}

\section{Introduction}	

Let $(M,g)$ be a closed Riemannian manifold of real dimension $4$ and let $E$ be a vector bundle over $M$ with structure group $G$, where $G$ is a compact Lie group. The Yang--Mills functional, defined on the space of connections of $E$, is given by
\begin{equation*} \label{YMF}
\mathcal{YM}(\nabla)=\frac{1}{2}\int_M|F_{\nabla}|^2d\mathrm{vol}_g,
\end{equation*}
where $\nabla$ is a metric compatiable connection, $F_{\nabla}$ denotes the curvature and the pointwise norm $|\cdot|$ is given by $g$ and the Killing form of $\mathrm{Lie}(G)$.

$\nabla$ is called a Yang--Mills connection of $E$ if it satisfies the Yang--Mills equation
\begin{equation*} \label{YME}
D^*_{\nabla}F_{\nabla}=0.
\end{equation*}
A solution of the Yang--Mills flow is given by a family of connections $\nabla_t:=\nabla(x,t)$ such that
\begin{equation*} \label{YMF}
\frac{\partial \nabla_t}{\partial t}=-D^*_{\nabla_t}F_{\nabla_t}\quad in \quad M\times [0,T).
\end{equation*}
The Yang--Mills flow was initially studied by Atiyah--Bott \cite{AB} and was suggested to understand the topology of the space of connections by infinite dimensional Morse theory.

We consider the Yang--Mills--Higgs $k$-functional (or Yang--Mills--Higgs $k$-energy)
\begin{equation} \label{YMHKF}
\mathcal{YMH}_k(\nabla,u)=\frac{1}{2}\int_M\Big[|\nabla^{(k)}F_{\nabla}|^2+|\nabla^{(k+1)}u|^2\Big]d\mathrm{vol}_g,
\end{equation}
where $k\in \mathbb{N}\cup \{0\}$, $\nabla$ is a connection on $E$ and $u$ is a section of $\Omega^0(\textmd{ad}E)$. In \cite{Z1}, we have considered the case when $u$ is a section of $\Omega^0(E)$. When $k=0$, \eqref{YMHKF} is the Yang--Mills--Higgs functional studied in \cite{H,HT}. In \cite{H}, Hassell proved the global existence of the Yang--Mills--Higgs flow in 3-dimensional Euclidean space. In \cite{HT}, Hong--Tian studied the global existence of Yang--Mills--Higgs flow in 3-dimensional Hyperbolic space, their results yields non-self dual Yang--Mills connections on $S^4$. In the new century, the study of Yang--Mills--Higgs flow has aroused a lot of attention (see \cite{Af,HT1,LZ,SW,Tr,Wi,ZZZ} and references therein).

The Yang--Mills--Higgs $k$-system, i.e. the corresponding Euler--Lagrange equations  of \eqref{YMHKF}, is
\begin{equation} \label{EL}
\begin{cases}
(-1)^kD^*_{\nabla}\Delta_{\nabla}^{(k)}F_{\nabla}+\sum\limits_{v=0}^{2k-1}P_1^{(v)}[F_{\nabla}]+P_2^{(2k-1)}[F_{\nabla}]\\
\quad \quad +\sum\limits_{i=0}^k\nabla^{*(i)}(\nabla^{(k+1)}u\ast \nabla^{(k-i)}u)=0,\\
\nabla^{*(k+1)}\nabla^{(k+1)}u=0,
\end{cases}
\end{equation}
where $\Delta_{\nabla}^{(k)}$ denotes $k$ iterations of the Bochner Laplacian $-\nabla^*\nabla$, and the notation $P$ is defined in \eqref{P}.

A solution of the Yang--Mills--Higgs $k$-flow is given by a family of pairs $(\nabla(x,t),u(x,t)):=(\nabla_t,u_t)$ such that
\begin{equation} \label{YMHKS}
\begin{cases}
\frac{\partial \nabla_t}{\partial t}&=(-1)^{(k+1)}D^*_{\nabla_t}\Delta_{\nabla_t}^{(k)}F_{\nabla_t}+\sum\limits_{v=0}^{2k-1}P_1^{(v)}[F_{\nabla_t}]\\
&~~+P_2^{(2k-1)}[F_{\nabla_t}]+\sum\limits_{i=0}^k\nabla_t^{*(i)}(\nabla_t^{(k+1)}u_t\ast \nabla_t^{(k-i)}u_t), \\
\frac{\partial u_t}{\partial t}&=-\nabla^{*(k+1)}_t\nabla_t^{(k+1)}u_t, \quad in \quad M\times [0,T).
\end{cases}
\end{equation}
When $k=0$, the flow \eqref{YMHKS} is a Yang--Mills--Higgs flow \cite{HT}.

From an analytic point of view, the Yang--Mills--Higgs $k$-flow \eqref{YMHKS} admits
similar properties to the case in which the Higgs field takes value in $\Omega^0(E)$.
In fact, by the approach in \cite{Z1}, we can prove the following theorem.

\begin{theorem} \label{mth1}
Let $E$ be a vector bundle over a closed Riemannian $4$-manifold $(M,g)$. Assume the integer $k> 1$, for every smooth initial data $(\nabla_0,u_0)$, there exists a unique smooth solution $(\nabla_t,u_t)$ to the Yang--Mills--Higgs $k$-flow \eqref{YMHKS} in $M\times [0,+\infty)$.
\end{theorem}

%The proof of the above theorem depends on local $L^2$-derivative estimates, energy %estimates and blow-up analysis.

%It is very interesting that, when we use a different %blow-up procedure, the above theorem can be improved when the the rank of $E$ equals %1. The following is our main result.

Our motivation for considering such flows comes from recent work of A. Waldron who was able to prove long time existence for the Yang--Mills flow \cite{Wa}, thereby settling a long standing conjecture in the area. In the context of the Yang--Mills--Higgs flow it is still unknown whether the flow exists for all times on a Riemannian 4-manifold.
The above theorem shows that provided $k > 1$, the Yang-Mills-Higgs $k$ flow does obey long time existence on a 4-manifold. A question that arises at this point is to understand what the optimum value for $k$ is.
By assuming our bundle $E$ is a line bundle, we are able to make progress on this question and show that long time existence holds for all positive $k$:

\begin{theorem} \label{mth2}
Let $E$ be a line bundle over a closed Riemannian $4$-manifold $(M,g)$. Assume the integer $k>0$, for every smooth initial data $(\nabla_0,u_0)$, there exists a unique smooth solution $(\nabla_t,u_t)$ to the Yang--Mills--Higgs $k$-flow \eqref{YMHKS} in $M\times [0,+\infty)$.
\end{theorem}

We point out that at present we don't know if the above theorem is optimal. Meaning we cannot rule out the case that long time existence occurs for $k=0$.

The proof of Theorem \ref{mth1} involves local $L^2$ derivative estimates, energy estimates and blowup analysis. An interesting aspect of this work is that by using
a different blow up procedure we are able to obtain a proof of Theorem \ref{mth2}, which may be of independent interest. Another interesting aspect is that the proof of long-time existence obstruction (see Theorem \ref{obst2}) is rather relevant to the properties of the Green function, which is very different from the previous techniques in \cite{Ke,Sa,Z1}.

\section{Preliminaries} \label{pre}

In this section we introduce the basic setup and notation that will be used throughout the paper. Our approach follows the notation of \cite{Ke}, \cite{Sa}, \cite{Z1}.

%To meet the requirements in the next sections, here, in this short section the setup and notation are briefly presented. We will use some of Kelleher's notations in \cite{Ke} and Saratchandran's in \cite{Sa} (see also \cite{Z1}).

Let $E$ be a vector bundle over a smooth closed manifold $(M,g)$ of real dimension $n$. The set of all smooth unitary connections on $E$ will be denoted by $\mathcal{A}_E$. A given connection $\nabla\in \mathcal{A}_E$ can be extended to other tensor bundles by coupling with the corresponding Levi--Civita connection $\nabla_M$ on $(M,g)$.

Let $D_{\nabla}$ be the exterior derivative, or skew symmetrization of $\nabla$. The curvature tensor of $E$ is denoted by
 $$F_{\nabla}=D_{\nabla}\circ D_{\nabla}.$$

We set $\nabla^{*},D_{\nabla}^{*}$ to be the formal $L^2$-adjoint of $\nabla,D_{\nabla}$, respectively. The Bochner and Hodge Laplacians are given respectively by
 $$\Delta_{\nabla}=-\nabla^{*}\nabla, \quad \Delta_{D_{\nabla}}=D_{\nabla}D^*_{\nabla}+D^*_{\nabla}D_{\nabla}.$$

Let $\xi,\eta$ be $p$-forms valued in $E$ or $\mathrm{End}(E)$. Let $\xi\ast \eta$ denote any mulitilinear form obtained from a tensor product $\xi\otimes\eta$ in a universal way. That is to say, $\xi\ast \eta$ is obtained by starting with $\xi\otimes\eta$, taking any linear combination of this tensor, taking any number of metric contractions, and switching any number of factors in the product. We then have $$|\xi\ast\eta|\leq C|\xi||\eta|.$$

Denote by
$$\nabla^{(i)}=\underbrace{\nabla\cdots\nabla}_{i\ \mathrm{times}}.$$

We will also use the $P$ notation, as introduced in \cite{KS}. Given a tensor $\xi$, we denote by
\begin{equation} \label{P}
P_v^{(k)}[\xi]:=\sum_{w_1+\cdots+w_v=k}(\nabla^{(w_1)}\xi)\ast\cdots\ast(\nabla^{(w_v)}\xi)\ast T,
\end{equation}
where $k,v\in \mathbb{N}$  and $T$ is a generic background tensor dependent only on $g$.

\vskip 1cm

\section{Long-time existence obstruction} \label{pot}

We can use De Turck's trick to establish the local existence of the Yang--Mills--Higgs $k$-flow. We refer to \cite{Ke,Sa,Z1} for more details. As the proof is standard, we will omit the details.

\begin{theorem} \label{le} (local existence)
Let $E$ be a vector bundle over a closed Riemannian manifold $(M,g)$. There exists a unique smooth solution $(\nabla_t,u_t)$ to the Yang--Mills--Higgs $k$-flow \eqref{YMHKS} in $M\times [0,\epsilon)$ with smooth initial value $(\nabla_0,u_0)$.
\end{theorem}

Following \cite{Ke,Sa}, we can derive estimates of Bernstein--Bando--Shi type, which is similar to \protect {\cite[Proposition 4.10]{Z1}}.

\begin{proposition} \label{bbs}
Let $q\in \mathbb{N}$ and $\gamma\in C^{\infty}_c(M)$ $(0\leq \gamma\leq 1)$. Suppose $(\nabla_t,u_t)$ is a solution to the Yang--Mills--Higgs $k$-flow \eqref{YMHKS} defined on $M\times I$. Suppose $Q=\max\{1,\sup\limits_{t\in I} |F_{\nabla_t}|\}$, $K=\max\{1,\sup\limits_{t\in I}|u_t|\}$, and choose $s\geq (k+1)(q+1)$. Then for $t\in [0,T)\subset I$ with $T<\frac{1}{(QK)^{4}}$, there exists a positive constant $C_q:=C_q(\mathrm{dim} (M),\mathrm{rk} (E),G,q,k,s,g,\gamma)\in \mathbb{R}_{>0}$ such that
\begin{equation} \label{bbsi}
\|\gamma^{s}\nabla_t^{(q)}F_{\nabla_t}\|_{L^2}^2+\|\gamma^{s}\nabla_t^{(q)}u_t\|_{L^2}^2
\leq C_qt^{-\frac{q}{k+1}}\sup_{t\in[0,T)}(\|F_{\nabla_t}\|^2_{L^2}+\|u_t\|^2_{L^2}).
\end{equation}
\end{proposition}

The following corollary is a direct consequence of the above proposition, which will be used in the blow-up analysis. The proof relies on Sobolev embedding $W^{p,2}\subset C^0$ provided $p>\frac{n}{2}$, and the Kato's inequality $|d|u_t||\leq |\nabla_tu_t|$. More details can be found in Kelleher's paper (\protect {\cite[Corollary 3.14]{Ke}}).

\begin{corollary} \label{ce1}
Suppose $(\nabla_t,u_t)$ solves the Yang--Mills--Higgs $k$-flow \eqref{YMHKS} defined on $M\times [0,\tau]$.  Set $\bar{\tau}:=\min\{\tau,1\}$. Suppose $Q=\max\{1,\sup\limits_{t\in [0,\bar{\tau}]} |F_{\nabla_t}|\}$, $K=\max\{1,\sup\limits_{t\in [0,\bar{\tau}]}|u_t|\}$. Assume $\gamma\in C^{\infty}_c(M)$ $(0\leq \gamma\leq 1)$. For $s,l\in \mathbb{N}$ with $s\geq (k+1)(l+1)$ there exists $C_l:=C_l(\dim (M),\mathrm{rk}(E),K,Q,G,s,k,l,\tau,g,\gamma)\in \mathbb{R}_{>0}$ such that
\begin{equation*}
\sup_M\Big(|\gamma^{s}\nabla_{\bar{\tau}}^{(l)}F_{\nabla_{\bar{\tau}}}|^2+|\gamma^{s}\nabla_{\bar{\tau}}^{(l)}u_{\bar{\tau}}|^2\Big)
\leq C_l\sup_{M\times [0,{\bar{\tau}})}\Big(\|F_{\nabla_t}\|^2_{L^2}+\|u_t\|^2_{L^2}\Big).
\end{equation*}
\end{corollary}

Using Corollary \ref{ce1}, we have the the following corollary, which can be used for finding obstructions to long time existence.

\begin{corollary} \label{ce2}
Suppose $(\nabla_t,u_t)$ solves the Yang--Mills--Higgs $k$-flow \eqref{YMHKS} defined on $M\times [0,T)$ for $T\in[0,+\infty)$.  Suppose
 $$Q=\max\{1,\sup\limits_{t\in [0,T)} |F_{\nabla_t}|,\sup\limits_{t\in [0,T)} \|F_{\nabla_t}\|_{L^2}\}$$
 and
 $$ K=\max\{1,\sup\limits_{t\in [0,T)}|u_t|,\sup\limits_{t\in [0,T)}\|u_t\|_{L^2}\}$$
 are finite. Assume $\gamma\in C^{\infty}_c(M)$ $(0\leq \gamma\leq 1)$.  Then for $t\in [0,T)$, $s,l\in \mathbb{N}$ with $s\geq (k+1)(l+1)$, there exists $C_l:=C_l(\nabla_0,u_0,\dim (M),\mathrm{rk}(E),K,Q,G,s,k,l,g,\gamma)\in \mathbb{R}_{>0}$ such that
\begin{equation*}
\sup_{M\times[0,T)}\Big(|\gamma^{s}\nabla_t^{(l)}F_{\nabla_{t}}|^2+|\gamma^{s}\nabla_t^{(l)}u_t|^2\Big)\leq C_l.
\end{equation*}
\end{corollary}

In the following, we will use Corollary \ref{ce2} to show that the only obstruction to long time existence of the Yang--Mills--Higgs $k$-flow \eqref{YMHKS} is a lack of supremal bound on $|F_{\nabla_t}|+|\nabla_tu_t|$. Before doing so, we need following proposition, which is similar to \protect {\cite[Proposition 4.15]{Z1}}.

\begin{proposition} \label{obst1}
Suppose $(\nabla_t,u_t)$ is a solution to the Yang--Mills--Higgs $k$-flow \eqref{YMHKS} defined on $M\times [0,T)$ for $T\in[0,+\infty)$. Suppose that for all $l\in \mathbb{N}\cup \{0\}$ there exists $C_l\in \mathbb{R}_{>0}$ such that
$$\max\Big\{\sup_{M\times[0,T)}|\nabla_t^{(l)}[\frac{\partial \nabla_t}{\partial t}]|, \sup_{M\times[0,T)}|\nabla_t^{(l)}[\frac{\partial u_t}{\partial t}]|\Big\}\leq C_l.$$
Then $\lim_{t\rightarrow T}(\nabla_t,u_t)=(\nabla_T,u_T)$ exists and is smooth.
\end{proposition}

The following proposition is straightforward.

\begin{proposition} \label{l2}
Suppose $(\nabla_t,u_t)$ is a solution to the Yang--Mills--Higgs $k$-flow \eqref{YMHKS} defined on $M\times [0,T)$. We have
$$
\sup_{t\in[0,T)}\|u_t\|_{L^2}<+\infty.
$$
\end{proposition}

Using Propositions \ref{obst1} and \ref{l2}, we are ready to prove the main result in this subsection.

\begin{theorem} \label{obst2}
Assume $E$ is a line bundle. Suppose $(\nabla_t,u_t)$ is a solution to the Yang--Mills--Higgs $k$-flow \eqref{YMHKS} for some maximal $T<+\infty$. Then
$$\sup_{M\times[0,T)}(|F_{\nabla_t}|+|\nabla_t u_t|)=+\infty.$$
\end{theorem}

\begin{proof}
Suppose to the contrary that
$$\sup_{M\times[0,T)}(|F_{\nabla_t}|+| \nabla_t u_t|)<+\infty,$$
which means
$$\sup_{M\times[0,T)}|F_{\nabla_t}|<+\infty, \quad \sup_{M\times[0,T)}| \nabla_t u_t|<+\infty.$$
Denote by $G_t(x,y)$ the Green function associated to the operator $\Delta_{\nabla_t}$, then for any fixed $x\in M$, $\|\nabla_0 G_t(x,\cdot)\|_{L^{\infty}(M)}\leq C_G$ for a constant $C_G$ \cite[Appendix A]{AS}. Note that $\nabla_tG_t-\nabla_0 G_t=[\nabla_t-\nabla_0,G_t]=0$, we conclude that $\|\nabla_t G_t\|_{L^{\infty}(M)}$ is also uniformly bounded. Therefore, using the properties of the Green function in \cite[Appendix A]{AS}, we have
\begin{equation*}
\begin{split}
|u_t(x)-\frac{1}{\textmd{Vol}(M)}\int_M u_t(y)\textmd{d}y|&=|\int_M \Delta_{\nabla_t} G_t(x,y) u_t(y)\textmd{d}y|\\
&=|\int_M \nabla_t G_t(x,y)\nabla_tu_t(y)\textmd{d}y|\\
&< +\infty,
\end{split}
\end{equation*}
which together with Proposition \ref{l2} implies
$$\sup_{M\times[0,T)}|u_t|<+\infty.$$

By Corollary \ref{ce2}, for all $t\in [0,T)$  and $l\in \mathbb{N}\cup\{0\}$, we have $\sup_{M}\Big(|\nabla_t^{(l)}F_{\nabla_{t}}|^2+|\nabla_t^{(l)}u_t|^2\Big)$ is uniformly bounded and so by Proposition \ref{obst1}, $\lim_{t\rightarrow T}(\nabla_t,u_t)=(\nabla_T,u_T)$ exists and is smooth. However, by local existence (Theorem \ref{le}), there exists $\epsilon>0$ such that $(\nabla_t,u_t)$ exists over the extended domain $[0,T+\epsilon)$, which contradicts the assumption that $T$ was maximal. Thus we prove the theorem.

\end{proof}

\vskip 1cm

\section{Blow-up analysis} \label{ba}

In this section, we will address the possibility that the Yang--Mills--Higgs $k$-flow
admits a singularity given no bound on $|F_{\nabla_t}|+|\nabla_tu_t|$. To begin with, we first establish some preliminary scaling laws for the Yang--Mills--Higgs $k$-flow.

\begin{proposition} \label{rf}
Suppose $(\nabla_t,u_t)$ is a solution to the Yang--Mills--Higgs $k$-flow \eqref{YMHKS} defined on $M\times [0,T)$. We define the 1-parameter family $\nabla_t^{\rho}$ with local coefficient matrices given by
$$\Gamma_t^{\rho}(x):=\rho \Gamma_{\rho^{2(k+1)}t}(\rho x),$$
where $\Gamma_t(x)$ is the local coefficient matrix of $\nabla_t$. We define the $\rho$-scaled Higgs field $u_t^{\rho}$ by
$$u_t^{\rho}(x):=\rho u_{\rho^{2(k+1)}t}(\rho x).$$
Then $(\nabla_t^{\rho},u_t^{\rho})$ is also a solution to the Yang--Mills--Higgs $k$-flow \eqref{YMHKS} defined on $[0,\frac{1}{\rho^{2(k+1)}}T)$.
\end{proposition}

Next we will show that in the case that the curvature coupled with a Higgs field is blowing up, as one approaches the maximal time, one can extract a blow-up limit. The proof will closely follow the arguments in \protect {\cite[Proposition 3.25]{Ke}} and \protect {\cite[Theorem 5.2]{Z1}}.

\begin{theorem} \label{bll}
Assume $E$ is a line bundle. Suppose $(\nabla_t,u_t)$ is a solution to the Yang--Mills--Higgs $k$-flow \eqref{YMHKS} defined on some maximal time interval $[0,T)$ with $T< +\infty$. Then there exists a blow-up sequence $(\nabla^i_t,u^i_t)$ and converges pointwise to a smooth solution $(\nabla^{\infty}_t,u^{\infty}_t)$ to the Yang--Mills--Higgs $k$-flow \eqref{YMHKS} defined on the domain $\mathbb{R}^n\times \mathbb{R}_{<0}$.
\end{theorem}

\begin{proof}
From Theorem \ref{obst2}, we must have
$$\lim_{t\rightarrow T}\sup_M\Big(|F_{\nabla_t}|+|\nabla_t u_t|\Big)=+\infty.$$
Therefore, we can choose a sequence of times $t_i\nearrow T$ within $[0,T)$, and a sequence of points $x_i$, such that
$$|F_{\nabla_{t_i}}(x_i)|+|\nabla_{t_i} u_{t_i}(x_i)| =\sup_{M\times [0,t_i]}\Big(|F_{\nabla_t}|+|\nabla_t u_t|\Big).$$
Let $\{\rho_i\}\subset \mathbb{R}_{>0}$ be constants to be determined. Define $\nabla_t^i(x)$ by
$$\Gamma_t^i(x)=\rho_i^{\frac{1}{2(k+1)}}\Gamma_{\rho_i t+t_i}(\rho_i^{\frac{1}{2(k+1)}} x+x_i)$$
and
$$u_t^i(x)=\rho_i^{\frac{1}{2(k+1)}}u_{\rho_i t+t_i}(\rho_i^{\frac{1}{2(k+1)}} x+x_i).$$
By Proposition \ref{rf}, $(\nabla^i_t,u^i_t)$ are also solutions to Yang--Mills--Higgs $k$-flow \eqref{YMHKS} and the domain for each $(\nabla^i_t,u^i_t)$ is $B_0(\rho_i^{-\frac{1}{2(k+1)}})\times [-\frac{t_i}{\rho_i},\frac{T-t_i}{\rho_i})$. We observe that
$$F_t^i(x):=F_{\nabla_t^i}(x)=\rho_i^{\frac{1}{k+1}}F_{\nabla_{\rho_i t+t_i}}(\rho_i^{\frac{1}{2(k+1)}}x+x_i),$$
which means
\begin{equation*}
\begin{split}
&\sup_{t\in[-\frac{t_i}{\rho_i},\frac{T-t_i}{\rho_i})}\Big(|F_t^i(x)|+|\nabla_{t}^iu_t^i(x)|\Big)\\
&=\rho_i^{\frac{1}{k+1}}\sup_{t\in[-\frac{t_i}{\rho_i},\frac{T-t_i}{\rho_i})}\Big(|F_{\nabla_{\rho_i t+t_i}}(\rho_i^{\frac{1}{2(k+1)}} x+x_i)|+|\nabla_{\rho_i t+t_i} u_{\rho_i t+t_i}(\rho_i^{\frac{1}{2(k+1)}} x+x_i)|\Big)\\
&=\rho_i^{\frac{1}{k+1}}\sup_{t\in[0,t_i]}\Big(|F_{\nabla_t}(x)|+|\nabla_tu_t(x)|\Big)\\
&=\rho_i^{\frac{1}{k+1}}\Big(|F_{\nabla_{t_i}}(x_i)|+|\nabla_{t_i}u_{t_i}(x_i)|\Big).
\end{split}
\end{equation*}
Therefore, setting
$$\rho_i=\Big(|F_{\nabla_{t_i}}(x_i)|+|\nabla_{t_i}u_{t_i}(x_i)|\Big)^{-(k+1)},$$
which gives
\begin{equation} \label{model}
1=|F^i_0(0)|+|\nabla_0^i u_0^i(0)|=\sup_{t\in[-\frac{t_i}{\rho_i},0]}\Big(|F_t^i(x)|+|\nabla_t^i u_t^i(x)|\Big).
\end{equation}

Now, we are ready to construct smoothing estimates for the sequence $(\nabla^i_t,u^i_t)$. Let $y\in \mathbb{R}^n$, $\tau\in \mathbb{R}_{\leq 0}$. For any $s\in \mathbb{N}$,
$$\sup_{t\in [\tau-1,\tau]}\Big(|\gamma^s_y F^i_t(x)|+|\gamma^s_y \nabla_t^i u^i_t(x)|\Big)\leq 1.$$
Note that $E$ is a line bundle, and similar to the proof of Theorem \ref{obst2}, it suffices to use Corollary \ref{ce1}. Then for all $q\in \mathbb{N}$, one may choose $s\geq(k+1)(q+1)$ so that there exists positive constant $C_q$ such that
\begin{equation*}
\begin{split}
&\sup_{x\in B_y(\frac{1}{2})}\Big(|(\nabla^i_{\tau})^{(q)} F^i_{\tau}(x)|+|(\nabla^i_{\tau})^{(q)}  u^i_{\tau}(x)|\Big)\\
&\leq\sup_{x\in B_y(1)}\Big(|\gamma^s_y(\nabla^i_{\tau})^{(q)} F^i_{\tau}(x)|+|\gamma^s_y(\nabla^i_{\tau})^{(q)}  u^i_{\tau}(x)|\Big)\\
&\leq C_q.
 \end{split}
\end{equation*}
Then by the Coulomb Gauge Theorem of Uhlenbeck \cite[Theorem 1.3]{U} (also see \cite{HT}) and Gauge Patching Theorem \cite[Corollary 4.4.8]{DK}, passing to a subsequence (without changing notation) and in an appropriate gauge, $(\nabla^i_t,u^i_t)\rightarrow (\nabla^{\infty}_t,u^{\infty}_t)$ in $C^{\infty}$.

\end{proof}

\vskip 1cm

\section{Proof of Theorem \ref{mth2}} \label{pot}

The following energy estimates are similar to the ones in \cite[Section 6]{Z1}.

\begin{proposition} \label{hoymhe}
Suppose $(\nabla_t,u_t)$ is a solution to the Yang--Mills--Higgs $k$-flow \eqref{YMHKS} defined on $M\times [0,T)$. The Yang--Mills--Higgs $k$-energy \eqref{YMHKF} is decreasing along the flow \eqref{YMHKS}.
\end{proposition}

\begin{proposition} \label{ymhe}
Suppose $(\nabla_t,u_t)$ is a solution to the Yang--Mills--Higgs $k$-flow \eqref{YMHKS} defined on $M^4\times [0,T)$ with $T< +\infty$, then the Yang--Mills--Higgs energy
$$
\mathcal{YMH}(\nabla_t,u_t)=\frac{1}{2}\int_M\Big[|F_{\nabla_t}|^2+|\nabla_t u_t|^2\Big]d\mathrm{vol}_g
$$
is bounded along the flow \eqref{YMHKS}.
\end{proposition}

Next, we will complete the proof of Theorem \ref{mth2}. To accomplish this, we first show that the $L^p$-norm controls the $L^{\infty}$-norm by blow-up analysis.

\begin{proposition} \label{lptli}
Assume $E$ is a line bundle. Suppose $(\nabla_t,u_t)$ is a solution to the Yang--Mills--Higgs $k$-flow \eqref{YMHKS} defined on $M^4\times [0,T)$ and
$$\sup_{t\in[0,T)}(\|F_{\nabla_t}\|_{L^p}+\|\nabla_t u_t\|_{L^p})<+\infty.$$
If $p>2$, then
$$\sup_{t\in[0,T)}(\|F_{\nabla_t}\|_{L^{\infty}}+\|\nabla_t u_t\|_{L^{\infty}})<+\infty.$$
\end{proposition}

\begin{proof}
So as to the obtain a contradiction, assume
 $$\sup_{t\in[0,T)}(\|F_{\nabla_t}\|_{L^{\infty}}+\|\nabla_t u_t\|_{L^{\infty}})=+\infty.$$
As we did in Theorem \ref{bll}, we can construct a blow-up sequence $(\nabla^i_t,u^i_t)$, with blow-up limit $(\nabla^{\infty}_t,u^{\infty}_t)$. Noting that \eqref{model}, by Fatou's lemma and natural scaling law,
\begin{equation*}
\begin{split}
\|F_{\nabla^{\infty}_t}\|^p_{L^p}+\|\nabla_t^{\infty} u^{\infty}_t \|^p_{L^p}
&\leq \lim_{i\rightarrow +\infty}\inf(\|F_{\nabla^{i}_t}\|^p_{L^p}+\|\nabla_t^i u^{i}_t \|^p_{L^p})\\
&\leq \lim_{i\rightarrow +\infty}\rho_i^{\frac{2p-4}{2k+2}}(\|F_{\nabla_t}\|^p_{L^p}+\|\nabla_t u_t\|^p_{L^p}).
\end{split}
\end{equation*}
Since $\lim_{i\rightarrow +\infty}\rho_i^{\frac{2p-4}{2k+2}}=0$ when $p>2$, the right hand side of the above inequality tends to zero, which is a contradiction since the blow-up limit has non-vanishing curvature.
\end{proof}

Now we are ready to give the proof of Theorem \ref{mth2}.

{\bf Proof of Theorem \ref{mth2}}.  By the Sobolev embedding theorem, we solve for $p$ such that $W^{k,2}\subset L^{p}$, then $k>0$. In this case, using the interpolation inequalities \cite[Corollary 5.5]{KS} we have
\begin{equation*}
\begin{split}
&\|F_{\nabla_t}\|_{L^{p}}+\|\nabla_t u_t\|_{L^{p}}\\
&\quad \leq CS_{k,p}\sum_{j=0}^k(\|\nabla_t^{(j)}F_{\nabla_t}\|^2_{L^2}+\|\nabla_t^{(j)}u_t\|^2_{L^2}+1)\\
&\quad \leq CS_{k,p}(\|\nabla_t^{(k)}F_{\nabla_t}\|^2_{L^2}+\|F_{\nabla_t}\|^2_{L^2}+\|\nabla_t^{(k+1)}u_t\|^2_{L^2}+\|u_t\|^2_{L^2}+1)\\
&\quad \leq CS_{k,p}(\mathcal{YMH}_k(\nabla_t,u_t)+\mathcal{YMH}(\nabla_t,u_t)+\|u_t\|^2_{L^2}+1).
\end{split}
\end{equation*}
Then using Propositions \ref{hoymhe}, \ref{l2} and \ref{ymhe}, we conclude that the flow exists smoothly for all time.

\vskip 1cm

\section*{Acknowledgments}

HS was supported by the Australian Research Council via grant FL170100020. PZ was supported by the Natural Science Foundation of Anhui Province [Grant Number 2108085QA17].

\end{document}